\documentclass[12pt]{amsart}
\usepackage{amsmath, amssymb, amscd, amsbsy, amsthm,braket}
\usepackage{graphicx}
\usepackage{xcolor}
\usepackage[all]{xy}
\usepackage{mathrsfs}
\usepackage{stmaryrd}

\newtheorem{theorem}{Theorem}[section]
\newtheorem{lemma}[theorem]{Lemma}
\newtheorem{proposition}[theorem]{Proposition}
\newtheorem{corollary}[theorem]{Corollary}
\theoremstyle{definition}
\newtheorem{definition}[theorem]{Definition}

\theoremstyle{remark}
\newtheorem{remark}[theorem]{Remark}
\numberwithin{equation}{section}

\def\Im{\operatorname{Im}}
\def\Ker{\operatorname{Ker}}
\def\GL{\operatorname{GL}}
\def\rank{\operatorname{rank}}
\def\Sp{\operatorname{Sp}}
\def\urSp{\operatorname{urSp}}
\def\Sign{\operatorname{Sign}}

\def\Hom{\operatorname{Hom}}
\def\id{\operatorname{id}}

\def\Z{{\mathbb Z}}
\def\Q{{\mathbb Q}}
\def\R{{\mathbb R}}

\def\HH{\mathscr{H}}

\def\mcg{\operatorname{Mod}(\Sigma_g)}
\def\hyp{\mathcal{H}(\Sigma_g)}
\def\hb{\operatorname{Mod}(V_g)}
\def\hhb{\mathcal{H}(V_g)}

\begin{document}
\title[]{The Meyer function on the handlebody group}

\author{Yusuke Kuno}
\address{
Department of Mathematics,
Tsuda University,
2-1-1 Tsuda-machi, Kodaira-shi, Tokyo 187-8577, Japan}
\email{kunotti@tsuda.ac.jp}

\author{Masatoshi Sato}
\address{
Department of Mathematics,
School of Science \& Technology for Future Life, 
Tokyo Denki University,
5 Senjyuasahi-cho, Adachi-ku, Tokyo 120-8551, Japan}
\email{msato@mail.dendai.ac.jp}

\date{\today}

\begin{abstract}
We give an explicit formula for the signature of handlebody bundles over the circle in terms of the homological monodromy.
This gives a cobounding function of Meyer's signature cocycle on the mapping class group of a $3$-dimensional handlebody, i.e., the handlebody group.
As an application, we give a topological interpretation for the generator of the first cohomology group of the hyperelliptic handlebody group.
\end{abstract}

\maketitle

\section{Introduction} \label{sec:intro}

Let $\Sigma_g$ be a closed connected oriented surface of genus $g$
and $\mcg$ the mapping class group of $\Sigma_g$,
namely the group of isotopy classes of orientation-preserving self-diffeomorphisms of $\Sigma_g$.
Unless otherwise stated, we assume that (co)homology groups have coefficients in $\Z$.
The second cohomology of $\mcg$ has been determined for all $g\ge 1$ by works of many people, in particular by the seminal work of Harer~\cite{Harer,Harer2} for $g\ge 3$.
We have $H^2(\operatorname{Mod}(\Sigma_1)) \cong \Z/ 12\Z$,
$H^2(\operatorname{Mod}(\Sigma_2)) \cong \Z/10 \Z$, and
\[
H^2(\mcg) \cong \Z \quad \text{for $g\ge 3$.}
\]
There are various interesting constructions of non-trivial second cohomology class of $\mcg$; the reader is referred to the survey article \cite{Kawazumi_survey}.
Among others, the remarkable approach of Meyer~\cite{Meyer2,Meyer} 
was to consider the signature of $\Sigma_g$-bundles over surfaces.
The central object that Meyer used was a normalized 2-cocycle
\[
\tau_g\colon \Sp(2g;\Z) \times \Sp(2g;\Z) \to \Z
\]
on the integral symplectic group of degree $2g$.

Meyer showed that for $g\ge 3$ the pullback of the cohomology class of $\tau_g$ by the homology representation $\rho\colon \mcg \to \Sp(2g;\Z)$ is of infinite order in $H^2(\mcg)$.
On the other hand, if $g = 1,2$ then $[\rho^*\tau_g]$ is torsion and
there exists a (unique) rational valued cobounding function $\phi_g\colon \mcg \to \Q$ of $\rho^*\tau_g$.
This means that
\[
\tau_g(\rho(\varphi_1),\rho(\varphi_2)) = \phi_g(\varphi_1) + \phi_g(\varphi_2) - \phi_g(\varphi_1\varphi_2)
\quad \text{for any $\varphi_1,\varphi_2 \in \mcg$.}
\]
Since the case $g=1$ was extensively studied by Meyer, such a cobounding function is called a Meyer function.
Some number-theoretic and differential geometric aspects of the function $\phi_1$ were studied by Atiyah \cite{Atiyah}.
The case $g = 2$ was studied by Matsumoto \cite{Matsumoto}, Morifuji \cite{Morifuji03} and Iida \cite{Iida}.
For $g\ge 3$, there is no cobounding function of $\rho^*\tau_g$ on the whole mapping class group $\mcg$.
However, if we restrict $\rho^*\tau_g$ to a subgroup called the hyperelliptic mapping class group $\mathcal{H}(\Sigma_g)$, then it is known that there is a (unique) cobounding function $\phi^{\mathcal{H}}_g \colon \mathcal{H}(\Sigma_g) \to \Q$ of $\rho^*\tau_g$.
Note that $\mathcal{H}(\Sigma_g) = \mcg$ for $g = 1,2$.
This function $\phi^{\mathcal{H}}_g$ was studied by Endo \cite{Endo00} and Morifuji \cite{Morifuji03}.
One motivation for studying Meyer functions comes from the localization phenomenon of the signature of fibered $4$-manifolds.
See, e.g., \cite{AshiEn, Kuno}.

In this paper, we study a new example of Meyer functions:
the Meyer function on the handlebody group. 
The handlebody group of genus $g$,
which we denote by $\hb$, is defined as 
the group of isotopy classes of orientation-preserving self-diffeomorphisms of the 3-dimensional handlebody $V_g$ of genus $g$.
It is well known that the natural homomorphism $\hb \to \mcg, \varphi \mapsto \varphi|_{\Sigma_g}$ is injective since $V_g$ is an irreducible 3-manifold.
Therefore, we can think of $\hb$ as a subgroup of $\mcg$.
For a mapping class $\varphi\in \hb$,
we denote by $M_{\varphi}$ the mapping torus of $\varphi$. It is a compact oriented $4$-manifold. 
We define
\[
\phi_g^V(\varphi):=\Sign M_{\varphi}\in\Z.
\]
We show in Lemma~\ref{lem:cobounding} that $\phi_g^V$ is a cobounding function of the cocycle $\rho^*\tau_g$
on the handlebody group $\hb$.
If $g\ge3$,  this is the unique cobounding function since $H_1(\hb)$ is torsion (see \cite[Theorem~20]{Wajnryb} and \cite[Remark~3.5]{IshidaSato}).

The value $\phi_g^V(\varphi)$ can be computed from the action of $\varphi$ on the first homology $H_1(\Sigma_g)$, 
and our first result gives its explicit description.
To state it, we take a suitable basis of $H_1(\Sigma_g)$ so that the homology representation $\rho$ restricted to $\hb$ takes values in a subgroup $\urSp(2g;\Z) \subset \Sp(2g;\Z)$.
(See Section \ref{subsec:hbg} for details.)
Then, $\rho(\varphi)$ is of the form
$\rho(\varphi) = \begin{pmatrix} P & Q \\ O_g & S \end{pmatrix}$, where $P$, $Q$ and $S$ are $g\times g$ matrices.
We consider a $\Q$-linear space $U_{\varphi}:=\Ker(S-I_g) \subset \Q^g$, and define a bilinear form $\langle \ ,\ \rangle_{\varphi}$ on it by
\[
\langle x, y \rangle_{\varphi} := {}^t x\,  {}^t Q\, y, \quad \text{for $x,y \in U_{\varphi}$}.
\]
It turns out that $\langle \ ,\ \rangle_{\varphi}$ is symmetric, and we have the following:

\begin{theorem}
\label{thm:main1}
The value $\phi_g^V(\varphi)$ coincides with the signature of the symmetric bilinear form $\langle \ , \ \rangle_{\varphi}$ on $U_{\varphi}$.
\end{theorem}
In fact, we will show in Section \ref{subsec:intform} that the intersection form on $H_2(M_{\varphi})$ is equivalent to the bilinear form $\langle \ ,\ \rangle_{\varphi}$.

As a corollary, we see that the function $\phi_g^V$ is bounded by $g=\rank H_1(V_g)$.
We also give sample calculations of $\phi_g^{V}$ in Lemmas~\ref{lem:phiVt} and \ref{lem:phiVs}.
%In \cite[p.\ 124]{Walker},
Walker also constructed a function $j\colon \mcg\to\Q$ whose restriction to $\hb$ coincides with $\phi_g^{V}$.
Our description of $\phi_g^V$ in Theorem~\ref{thm:main1} is similar to but different from a description of $j$ given by Gilmer and Masbaum~\cite[Proposition~6.9]{GilMas}. 
See, for details, Remark~\ref{rem:GilMas}.

As an application of the function $\phi_g^V$,
we obtain a non-trivial first cohomology class in the intersection $\hyp\cap\hb$ called the hyperelliptic handlebody group, denoted by $\hhb$.
The group $\hhb$ is an extension by $\mathbb{Z}/ 2\mathbb{Z}$ of a subgroup of 
the mapping class group of a 2-sphere with $(2g + 2)$-punctures, called the Hilden group.
The Hilden group was introduced in~\cite{Hilden},
and it is related to the study of links in $3$-manifolds.
In \cite{HiroseKin}, Hirose and Kin studied the minimal dilatation of pseudo-Anosov elements in $\hhb$, and gave a presentation of $\hhb$.

We consider the difference
\[
\phi_g^{\mathcal{H}}-\phi_g^V\in \Hom(\hhb,\mathbb{Q})=H^1(\hhb;\Q)
\]
of the Meyer functions on $\hyp$ and on $\hb$.
From the abelianization of $\hhb$ obtained in \cite[Corollary~A.9]{HiroseKin},
we see that the rank of $H^1(\hhb)$ is one.
Let us denote a generator of $H^1(\hhb)$ by $\mu$.
Our second result is:
\begin{theorem}
\label{thm:main}
Let $g\ge 1$. We have
\[
\phi_g^{\mathcal{H}}-\phi_g^V
=\begin{cases}
\displaystyle\frac{2}{2g + 1}\mu&\text{if }g\text{ is even},\\[10pt]
\displaystyle\frac{1}{2g + 1}\mu&\text{if }g\text{ is odd}.
\end{cases}
\]
\end{theorem}
When $g = 1,2$, we have $\hhb=\hb$,
and $\phi_g^{\mathcal{H}}-\phi_g^V$
gives an abelian quotient of $\hb$.

There is an interpretation of the cohomology class 
$\phi_g^{\mathcal{H}}-\phi_g^V$
in terms of a kind of connecting homomorphism.
We assume that $g\ge3$.
From the diagram
\[
\begin{CD}
\hhb@>i_2>>\hb\\
@Vi_1VV@VVj_2V\\
\hyp@>>j_1>\mcg.
\end{CD}
\]
of groups and their inclusions,
we have a natural homomorphism
\[
\Upsilon\colon H^2(\mcg;\mathbb{Q})\to H^1(\hhb;\mathbb{Q})
\]
defined as follows.
For $[c]\in H^2(\mcg;\mathbb{Q})$,
there are cobounding functions $f^{\mathcal{H}}\colon\hyp\to\mathbb{Q}$ of $j_1^*c$ and $f^V\colon\hb\to\mathbb{Q}$ of $j_2^*c$, respectively.
The cochain $i_1^*f^{\mathcal{H}}-i_2^*f^V$ is actually a homomorphism on  $\hhb$.
It does not depend on the choices of the representatives $c$, $f^{\mathcal{H}}$, and $f^V$
since $H^1(\hb;\mathbb{Q})=H^1(\hyp;\mathbb{Q})=0$ when $g\ge3$.
Then $\Upsilon([c])$ is defined to be $i_1^*f^{\mathcal{H}}-i_2^*f^V$.
In this setting, our cohomology class is written as $\Upsilon([\tau_g])=\phi_g^{\mathcal{H}}-\phi_g^V\in H^1(\hhb;\Q)$.

The outline of this paper is as follows.
In Section~\ref{sec:pre},
we review the definition of Meyer's signature cocycle and the handlebody group $\hb$.
We also review the abelianization of the hyperelliptic handlebody group obtained in \cite{HiroseKin},
and describe a generator of the cohomology group $H^1(\hhb)$ in Corollary~\ref{cor:mu}.
In Section~\ref{sec:hbb}, we investigate the intersection form
of the mapping torus of $\varphi \in \hb$,
and prove Theorem~\ref{thm:main1}.
As it turns out, we can explicitly describe $\phi^V_g$ as a function on a subgroup $\urSp(2g;\mathbb{Z})$ of the integral symplectic group.
In Section~\ref{section:Meyer function},
we prove Theorem~\ref{thm:main} by using explicit calculations of the Meyer function $\phi_g^{V}\colon\hb\to\mathbb{Z}$ in Lemmas~\ref{lem:phiVt} and \ref{lem:phiVs}.

\section{Preliminaries on mapping class groups} \label{sec:pre}

Fix a non-negative integer $g$.

\subsection{Mapping class group of a surface} \label{subsec:mcg}

Let $\Sigma_g$ be a closed connected oriented surface of genus $g$.
The \emph{mapping class group} of $\Sigma_g$, denoted by $\mcg$, is the group of isotopy classes of orientation-preserving self-diffeomorphisms of $\Sigma_g$.
To simplify notation, we will use the same letter for a self-diffeomorphism of $\Sigma_g$ and its isotopy class.

The first homology group $H_1(\Sigma_g)$ is equipped with a non-degenerate skew-symmetric pairing $\langle  \cdot, \cdot \rangle$, namely the intersection form.
Thus we can take a symplectic basis $\alpha_1,\ldots,\alpha_g,\beta_1,\ldots,\beta_g$ for $H_1(\Sigma_g)$.
This means that $\langle \alpha_i, \beta_j \rangle = \delta_{ij}$ and $\langle \alpha_i,\alpha_j \rangle = \langle \beta_i,\beta_j \rangle = 0 $ for any $i,j\in \{1,\ldots,g\}$, where $\delta_{ij}$ is the Kronecker symbol.
%We may assume that $\alpha_i$ and $\beta_i$ are represented by oriented simple closed curves in $\Sigma_g$.

Once a symplectic basis for $H_1(\Sigma_g)$ is fixed,
we obtain the homology representation
\[
\rho\colon \mcg \to \Sp(2g;\Z), \quad \varphi \mapsto \varphi_*.
\]
Here, the target is the integral \emph{symplectic group}
\[
\Sp(2g;\Z) = \{ A \in \GL(2g;\Z) \mid {}^t\!A J A = J \},
\]
where $J = \begin{pmatrix} O_g & I_g \\ -I_g & O_g \end{pmatrix}$,
and $\rho(\varphi)=\varphi_*$ is the matrix presentation of the action of $\varphi$ on $H_1(\Sigma_g)$ with respect to the fixed symplectic  basis.
We use block matrices to denote elements in $\Sp(2g;\Z)$, e.g., $A= \begin{pmatrix} P & Q \\ R & S \end{pmatrix}$ with $g\times g$ integral matrices $P$, $Q$, $R$, and $S$.

\subsection{Meyer's signature cocycle} \label{subse:taug}

Let $A,B\in \Sp(2g;\Z)$.
We consider an $\R$-linear space
\[
V_{A,B}:= \{ (x,y) \in \R^{2g} \oplus \R^{2g} \mid
(A^{-1} - I_{2g}) x + (B-I_{2g}) y = 0\}
\]
and a bilinear form on $V_{A,B}$ given by
\[
\langle (x,y), (x',y') \rangle_{A,B}:=
{}^t (x + y)J(I_{2g} - B)y'.
\]

The form $\langle \cdot, \cdot \rangle_{A,B}$ turns out to be symmetric, and thus its signature is defined; we set
\[
\tau_g(A,B):= \Sign(V_{A,B}, \langle \cdot, \cdot \rangle_{A,B}).
\]
The map $\tau_g\colon \Sp(2g;\Z) \times \Sp(2g;\Z) \to \Z$ is called \emph{Meyer's signature cocycle} \cite{Meyer2,Meyer}. 
It is a normalized 2-cocycle of the group $\Sp(2g;\Z)$.

Let $P$ be a compact oriented surface of genus $0$ with three boundary components, i.e., a pair of pants.
We denote by $C_1$, $C_2$ and $C_3$ the boundary components of $P$.
Choose a base point in $P$,
and let $\ell_1$, $\ell_2$ and $\ell_3$ be based loops in $P$ such that $\ell_i$ is parallel to the negatively oriented boundary component $C_i$ for any $i \in \{1,2,3\}$
and $\ell_1\ell_2\ell_3 = 1$ holds in the fundamental group $\pi_1(P)$.

For given two mapping classes $\varphi_1, \varphi_2 \in \mcg$, there is an oriented $\Sigma_g$-bundle $E(\varphi_1,\varphi_2) \to P$ such that the monodromy along $\ell_i$ is $\varphi_i$ for $i = 1,2$.
It is unique up to bundle isomorphisms. 
The total space $E(\varphi_1,\varphi_2)$ is a compact 4-manifold equipped with a natural orientation,
and hence its signature is defined.

\begin{proposition}[Meyer \cite{Meyer2,Meyer}]
\label{prop:Meyer}
$\Sign(E(\varphi_1,\varphi_2)) = \tau_g(\rho(\varphi_1),\rho(\varphi_2))$.
\end{proposition}

\begin{remark}
\label{rem:Maslov}
Turaev \cite{Turaev} independently found the signature cocycle.
He also studied its relation to the Maslov index.
\end{remark}

\subsection{Handlebody group}
\label{subsec:hbg}

Let $V_g$ be a handlebody of genus $g$.
That is, $V_g$ is obtained by attaching $g$ one-handles to the $3$-ball $D^3$.
We identify $\Sigma_g$ and the boundary of $V_g$ by choosing an orientation-preserving diffeomorphism between them.
We have the following short exact sequence
\begin{equation}
\label{eq:VSigma}
0 \longrightarrow H_2(V_g,\Sigma_g) \overset{\partial_*}{\longrightarrow} H_1(\Sigma_g) \overset{i_*}{\longrightarrow} H_1(V_g) \longrightarrow 0
\end{equation}
which is a part of the homology exact sequence of the pair $(V_g,\Sigma_g)$.
There are properly embedded, oriented and pairwise disjoint disks $D_1,\ldots,D_g$ in $V_g$ whose homology classes (denoted by the same letters) constitute a basis for $H_2(V_g,\Sigma_g)$.
We set $\alpha_i := \partial_*(D_i) \in H_1(\Sigma_g)$ for $i\in \{1,\ldots,g\}$.
Then $\alpha_i$'s extend to a symplectic basis $\alpha_1,\ldots,\alpha_g,\beta_1,\ldots,\beta_g$ for $H_1(\Sigma_g)$.
In what follows, we fix a symplectic basis obtained in this way.
%We fix a basis $D_1,\ldots,D_g$ for $H_2(V_g,\Sigma_g)$ such that $\partial_*(D_i) = \alpha_i$ for any $i\in \{1,\ldots,g\}$.
%Topologically, the homology classes $D_i$ are realized as properly embedded, oriented and pairwise disjoint $g$ disks in $V_g$.
The image of the homology classes $\beta_1, \ldots, \beta_g$ by the map $i_*$ constitute a basis for $H_1(V_g)$.
For simplicity, we denote them by the same letters $\beta_1,\ldots,\beta_g$.

%Let $V_g$ be a handlebody of genus $g$.
%That is, $V_g$ is obtained by attaching $g$ one-handles to the $3$-ball $D^3$.
%We identify $\Sigma_g$ and the boundary of $V_g$ by choosing an orientation-preserving diffeomorphism between them.
%By a suitable choice, we can arrange that there are properly embedded oriented $g$ disks $D_1,\ldots,D_g$ in $V_g$ such that $D_i$ is bounded by the simple closed curve $\alpha_i$ for any $i\in \{1,\ldots,g\}$ and $D_i \cap D_j = \emptyset$ for any $i\neq j$.
%We have the following short exact sequence
%\begin{equation}
%\label{eq:VSigma}
%0 \longrightarrow H_2(V_g,\Sigma_g) \overset{\partial_*}{\longrightarrow} H_1(\Sigma_g) \overset{i_*}{\longrightarrow} H_1(V_g) \longrightarrow 0
%\end{equation}
%which is a part of the homology exact sequence of the pair $(V_g,\Sigma_g)$.
%The disks $D_1,\ldots,D_g$ constitute a basis for $H_2(V_g,\Sigma_g)$,
%and the simple closed curves $\beta_1,\ldots,\beta_g$ constitute a basis for $H_1(V_g)$.
%Note that $\partial_* (D_i) = \alpha_i$ for any $i\in \{ 1,\ldots,g\}$.

We denote by $\hb$ the \emph{handlebody group} of genus $g$.
It can be considered as a subgroup of $\mcg$.
For any $\varphi \in \hb$, the matrix $\rho(\varphi)$ lies in the subgroup of $\Sp(2g;\Z)$ defined by
\[
\urSp(2g;\Z) :=
\left\{
A \in \Sp(2g;\Z) \mid A= \begin{pmatrix} P & Q \\ O_g & S \end{pmatrix}
\right\},
\]
cf. \cite{Birman75, Hirose06} for details.
The matrices $P$, $Q$ and $S$ satisfy the following relations:
\begin{equation}
\label{eq:PQS}
{}^t\!P S =I_g, \quad
{}^t Q S = {}^t\! S Q.
\end{equation}

\begin{remark}
\label{rem:hbhom}
The group $\hb$ acts naturally on the groups in \eqref{eq:VSigma},
and the maps $\partial_*$ and $i_*$ are $\hb$-module homomorphisms.
The matrix presentation of the action $\varphi_*$ on $H_1(V_g)$ is $S$.
\end{remark}

\subsection{Hyperelliptic handlebody group}
\label{subsec:hhb}

An involution of $\Sigma_g$ is called \emph{hyperelliptic} if it acts on $H_1(\Sigma_g)$ as $-\id$.
We fix an hyperelliptic involution $\iota$ which extends to an involution of $V_g$, as in Figure~\ref{fig:curves}.
\begin{figure}[h]
\begin{center}
\includegraphics{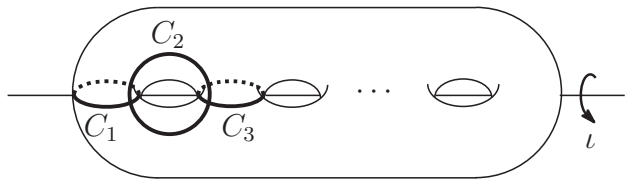}
\end{center}
\caption{The involution $\iota$ of $V_g$ and the curves $C_1$, $C_2$, $C_3$.}
\label{fig:curves}
\end{figure}

The \emph{hyperelliptic mapping class group} $\hyp$ is the centralizer of $\iota$ in $\mcg$:
\[
\hyp := \{ \varphi \in \mcg \mid \varphi \iota = \iota \varphi \}.
\]

\begin{definition}[\cite{HiroseKin}]
\label{dfn:hhb}
The \emph{hyperelliptic handlebody group} $\hhb$ is defined by
\[
\hhb:=\hyp \cap \hb.
\]
\end{definition}

Hirose and Kin \cite[Appendix~A]{HiroseKin} gave a finite presentation of the group $\hhb$.
Moreover they determined the abelianization of $\hhb$ as
\[
\hhb^{\rm abel} \cong \Z \oplus \Z_2 \oplus \Z_2
\quad \text{for $g \ge 2$}.
\]
In fact, using their presentation, it is easy to make this result more explicit.
Let $C_1$, $C_2$ and $C_3$ be simple closed curves on $\Sigma_g$ as in Figure~\ref{fig:curves}.
For each $i\in \{1,2,3\}$ denote by $t_i$ the \emph{right} handed Dehn twist along $C_i$. 
Following \cite{HiroseKin}, set $r_1= t_2^{-1} t_3^{-1} t_1 t_2$ and $s_1 = t_2 t_3 t_1 t_2$.
(Note that in \cite{HiroseKin}, $t_C$ denotes the \emph{left} handed Dehn twist along $C$.)

\begin{lemma}
\label{lem:abel}
When $g=1$, one has $\mathcal{H}(V_1) \cong \Z\, [t_1 s_1] \oplus \Z_2\, [t_1^2 s_1]$.
If $g \ge 2$, then
\[
\hhb^{\rm abel}
\cong
\begin{cases}
\Z\, [s_1] \oplus \Z_2\, [t_1 s_1^{\frac{g}{2}}] \oplus \Z_2\, [r_1] & \text{if $g$ is even}, \\
\Z\, [t_1 s_1^{\frac{g+1}{2}}] \oplus \Z_2\, [t_1^2 s_1^g] \oplus \Z_2\, [r_1] & \text{if $g$ is odd}.
\end{cases}
\]
Here, $[s_1]$ is the class of $s_1$ in $\hhb^{\rm abel}$, 
and $\Z\,[s_1]$ is the infinite cyclic group generated by $[s_1]$, etc.
\end{lemma}

\begin{proof}
The case $g=1$ follows from the fact that $\mathcal{H}(V_1) \cong \operatorname{Mod}(V_1)$ and a result of Wajnryb \cite[Theorem 14]{Wajnryb}.

Assume that $g\ge 2$.
Using \cite[Theorem A.8]{HiroseKin}, one sees that $\hhb^{\rm abel}$ is generated by $[r_1]$, $[s_1]$ and $[t_1]$ with the relations
\[
2 [r_1] =0, \quad
4 [t_1] + 2g [s_1] = 0, \quad
2(g + 1) [t_1] + g(g + 1) [s_1] = 0.
\]
The assertion follows from these relations by a direct computation.
\end{proof}

The following corollary to Lemma~\ref{lem:abel} will be used in Section~\ref{subsec:pr11} to prove Theorem~\ref{thm:main}.

\begin{corollary}
\label{cor:mu}
Let $g\ge 1$.
There is a unique homomorphism $\mu\colon \mathcal{H}(V_g) \to \Z$ satisfying the following property:
\begin{enumerate}
\item[$(1)$] If $g$ is even, $\mu(s_1)=1$ and $\mu(t_1)=-g / 2$;
\item[$(2)$] If $g$ is odd, $\mu(t_1)=-g$, $\mu(s_1)=2$, and thus $\mu(t_1s_1^{\frac{g+1}{2}})=1$.
\end{enumerate}
Moreover, the first cohomology group $H^1(\mathcal{H}(V_g)) = \Hom(\mathcal{H}(V_g),\Z)$ is an infinite cyclic group generated by $\mu$.
\end{corollary}

\section{Handlebody bundles over $S^1$} \label{sec:hbb}

\subsection{Mapping torus} \label{subsec:mtorus}

Let $I=[0,1]$ be the unit interval.
By identifying the endpoints of $I$, we obtain the circle $S^1 = [0,1]/ 0 \sim 1$.
Let $\ell\colon I \to S^1$ be the natural projection.
For $t\in I$, we set $[t]:= \ell(t)$.
Choose $[0]$ as a base point of $S^1$.
Then the fundamental group $\pi_1(S^1)$ is an infinite cyclic group generated by the homotopy class of $\ell$.

In what follows, we use the following cell decomposition of $S^1$: the $0$-cell is $e^0= [0]$ and the 1-cell is $e^1 = S^1 \setminus e^0$.
The map $\ell$ induces an orientation of $e^1$.

Let $\varphi \in \hb$.
The \emph{mapping torus} of $\varphi$ is the quotient space
\[
M_{\varphi}:= (I\times V_g)/ (0,x) \sim (1,\varphi(x)).
\]
For $(t,x) \in I \times V_g$, its class in $M_{\varphi}$ is denoted by $[t,x]$.
The natural projection
$
\pi\colon M_{\varphi} \to S^1, [t,x] \mapsto [t]
$
is an oriented $V_g$-bundle,
and the total space $M_{\varphi}$ is a compact 4-manifold with boundary equipped with a natural orientation.
The pullback of $M_{\varphi} \to S^1$ by $\ell$ is a trivial $V_g$-bundle over $I$,
and its trivialization is given by the map
\begin{equation}
\label{eq:Phi}
\Phi\colon I \times V_g \to M_{\varphi}, \quad
(t,x) \mapsto [t,x].
\end{equation}
The following composition of maps coincides with $\varphi$:
\[
V_g \overset{0\times \id}{\cong} \{0\} \times V_g \overset{\Phi(0,\cdot)}{\longrightarrow}
\pi^{-1}([0]) = \pi^{-1}([1])
\overset{\Phi(1,\cdot)^{-1}}{\longrightarrow}
\{ 1\} \times V_g \overset{1\times \id}{\cong} V_g.
\]
Therefore, the monodromy of $M_{\varphi} \to S^1$ along $\ell$ is equal to the mapping class $\varphi$.
As was mentioned in Remark~\ref{rem:hbhom}, the groups $H_2(V_g,\Sigma_g)$, $H_1(\Sigma_g)$, and $H_1(V_g)$ are $\hb$-modules.
Thus, these groups become $\pi_1(S^1)$-modules; the homotopy class of $\ell$, which is a generator of $\pi_1(S^1)$, acts as the monodromy $\varphi \in \hb$.
%Note that the groups $H_2(V_g,\Sigma_g)$, $H_1(\Sigma_g)$, and $H_1(V_g)$ become $\pi_1(S^1)$-modules through the monodromy $\varphi$,
%and the $\hb$-action on them which was mentioned in Remark~\ref{rem:hbhom}.

\subsection{Second homology of the mapping torus}
\label{subsec:2hom}

For a non-negative integer $q\ge 0$, let $\HH_q(V_g)$ be the local system on $S^1$ which comes from the $V_g$-bundle $\pi\colon M_{\varphi} \to S^1$,
and whose fiber at $x\in S^1$ is the $q$-th homology group $H_q(\pi^{-1}(x))$.
Similarly, we consider the local system $\HH_q(V_g,\Sigma_g)$ whose fiber at $x\in S^1$ is the $q$-th relative homology group $H_q(\pi^{-1}(x),\partial \pi^{-1}(x))$.

Consider the Serre homology spectral sequence of the $V_g$-bundle $M_{\varphi} \to S^1$.
It degenerates at the $E^2$ page, which is given by $E^2_{p,q} =H_p (S^1; \HH_q(V_g))$.
Since $H_2(V_g)=0$ and the base space $S^1$ is $1$-dimensional,
we obtain
\[
H_2(M_{\varphi}) \cong E^{\infty}_{1,1} \cong E^2_{1,1}
= H_1(S^1; \HH_1(V_g)).
\]
Moreover, using the cellular homology of $S^1$ with coefficients in $\HH_1(V_g)$, we have
\begin{align*}
H_1(S^1; \HH_1(V_g)) &\cong
\Ker (\partial \colon C_1(S^1; \HH_1(V_g)) \to
C_0(S^1; \HH_1(V_g))) \\
&= \Ker(\partial \colon \Z e^1 \otimes H_1(V_g) \to
\Z e^0 \otimes H_1(V_g) = H_1(V_g)),
\end{align*}
where the boundary map is given by
\[
\partial (e^1 \otimes \alpha) = \ell_*(\alpha) - \alpha
= (\Phi(0,\cdot)^{-1} \circ \Phi(1,\cdot))_* (\alpha) - \alpha
=\varphi^{-1}_* (\alpha) - \alpha.
\]
In summary, we have proved the following lemma.
In the statement, $H_1(V_g)^{\pi_1(S^1)}$ is the space of invariants under the action of $\pi_1(S^1)$, i.e., 
$
H_1(V_g)^{\pi_1(S^1)} = \{ \alpha \in H_1(V_g) \mid \varphi_*(\alpha) = \alpha \}
$.

\begin{lemma}
\label{lem:H2abs}
We have
$
H_2(M_{\varphi}) \cong H_1(S^1; \HH_1(V_g)) \cong H_1(V_g)^{\pi_1(S^1)}
$.
%Here, the right hand side is the space of invariants under the action of $\pi_1(S^1)$: $H_1(V_g)^{\pi_1(S^1)}=\{ \alpha \in H_1(V_g) \mid \varphi_*(\alpha) = \alpha \}$.
\end{lemma}

Similarly, for the relative homology of the pair $(M_{\varphi},\partial M_{\varphi})$,
there is a spectral sequence converging to $H_*(M_{\varphi},\partial M_{\varphi})$
such that $E^2_{p,q} = H_p(S^1; \HH_q(V_g,\Sigma_g))$. 
This degenerates at the $E^2$ page, too.
Since $H_1(V_g,\Sigma_g) = 0$, we obtain
\[
H_2(M_{\varphi},\partial M_{\varphi}) \cong E^{\infty}_{0,2} \cong E^2_{0,2}
= H_0(S^1; \HH_2(V_g,\Sigma_g)).
\]
By the same argument as above, we obtain the following lemma.
In the statement,
$H_2(V_g,\Sigma_g)_{\pi_1(S^1)}$ is the space of coinvariants under the action of $\pi_1(S^1)$, i.e.,
the quotient of $H_2(V_g,\Sigma_g)$ by the subgroup generated by the set $\{ \varphi_*(\delta) - \delta \mid \delta \in H_2(V_g,\Sigma_g)$\}.

\begin{lemma}
\label{lem:H2rel}
We have
$
H_2(M_{\varphi},\partial M_{\varphi}) \cong H_0(S^1; \HH_2(V_g,\Sigma_g)) \cong
H_2(V_g,\Sigma_g)_{\pi_1(S^1)}
$.
%Here, the right hand side is the space of coinvariants under the action of $\pi_1(S^1)$.
\end{lemma}

\subsection{Description of the inclusion homomorphism}
\label{subsec:incl}

Recall that the short exact sequence \eqref{eq:VSigma} is $\hb$-equivariant.
Let $\alpha \in H_1(V_g)^{\pi_1(S^1)}$ be a $\varphi_*$-invariant homology class.
Pick an element $\tilde{\alpha} \in H_1(\Sigma_g)$ such that $i_*(\tilde{\alpha}) = \alpha$.
Then $\varphi_*(\tilde{\alpha}) - \tilde{\alpha} \in \Ker (i_*) = \Im (\partial_*)$.

\begin{definition}
\label{dfn:dalpha}
$d(\alpha) := [\partial_*^{-1} ( \varphi_*(\tilde{\alpha}) - \tilde{\alpha})]
\in H_2(V_g,\Sigma_g)_{\pi_1(S^1)}$.
\end{definition}

It is easy to see that $d(\alpha)$ is independent of the choice of $\tilde{\alpha}$.
Thus we obtain a well-defined map $d\colon H_1(V_g)^{\pi_1(S^1)} \to H_2(V_g,\Sigma_g)_{\pi_1(S^1)}$.

\begin{proposition}
\label{prop:incl}
The following diagram is commutative:
\[
\xymatrix{
H_1(V_g)^{\pi_1(S^1)} \ar[r]^-{d} \ar[d]^{\cong} & H_2(V_g,\Sigma_g)_{\pi_1(S^1)} \ar[d]^{\cong} \\
H_2(M_{\varphi}) \ar[r]^-{i_*} & H_2(M_{\varphi},\partial M_{\varphi}),
}
\]
where the bottom horizontal arrow is the inclusion homomorphism,
and the vertical arrows are the isomorphisms in Lemmas~\ref{lem:H2abs} and \ref{lem:H2rel}.
\end{proposition}

\subsection{Proof of Proposition~\ref{prop:incl}} \label{subsec:pr34}

In this section, for a topological space $X$,
we denote by $S_n(X)$ and $Z_n(X)$ the groups of singular $n$-chains and singular $n$-cycles, respectively.

Let $\alpha \in H_1(V_g)^{\pi_1(S^1)}$.
Pick its lift $\tilde{\alpha} \in H_1(\Sigma_g)$ such that $i_*(\tilde{\alpha}) = \alpha$.
Take a singular $1$-cycle $\tilde{a}\in Z_1(\Sigma_g)$ representing the homology class $\tilde{\alpha}$.
Then, $\varphi^{-1}_{\sharp}(\tilde{a}) - \tilde{a}$ is a singular $1$-boundary in $V_g$ since $\varphi_*^{-1}(\tilde{\alpha}) - \tilde{\alpha}\in \Ker(i_*)$.
Therefore, there exists $\sigma_{\varphi,\alpha}\in S_2(V_g)$ such that $\partial \sigma_{\varphi,\alpha} = \varphi^{-1}_{\sharp}(\tilde{a}) - \tilde{a}$.

First we compute the composition of $d$ and the right vertical map.
We claim that $d(\alpha)$ is represented by the relative $2$-cycle $-\sigma_{\varphi,\alpha} \in Z_2(V_g,\Sigma_g)$.
This follows from the equality
$\varphi_*(\tilde{\alpha}) -\tilde{\alpha} = - (\varphi^{-1}_*(\tilde{\alpha}) - \tilde{\alpha})$ in $H_1(\Sigma_g)_{\pi_1(S^1)}$
and the relation $\partial \sigma_{\varphi,\alpha} = \varphi^{-1}_{\sharp}(\tilde{a}) -\tilde{a}$.
Hence, the right vertical map sends $d(\alpha)$ to the homology class represented by the relative $2$-cycle $-e^0 \times \sigma_{\varphi,\alpha}\in Z_2(M_{\varphi},\partial M_{\varphi})$, where the symbol $\times$ means the cross product.

Next we compute the composition of the left vertical map and $i_*$.
For this purpose, we set
\[
\mathcal{Z}_{\alpha} := \Phi_{\sharp} ( I \times \tilde{a} ) - e^0 \times \sigma_{\varphi,\alpha} \in S_2(M_{\varphi}).
\]
Here, $\Phi$ is the map defined in \eqref{eq:Phi}, and the unit interval is regarded as a singular $1$-chain in the obvious way.
Actually, $\mathcal{Z}_{\alpha}$ is a $2$-cycle in $M_{\varphi}$.

\begin{lemma}
\label{lem:aZa}
The isomorphism in Lemma~\ref{lem:H2abs} sends $\alpha$ to the homology class of $\mathcal{Z}_{\alpha}$.
\end{lemma}

\begin{proof}
We need to inspect the spectral sequence involved in Lemma~\ref{lem:H2abs}.
For simplicity we denote $M=M_{\varphi}$,
and for every non-negative integer $q\ge 0$
let $M^{(q)}$ be the inverse image of the $q$-skeleton of $S^1$ by the projection map $\pi$.
Thus we have $\emptyset \subset M^{(0)} = \pi^{-1}([0]) \subset M^{(1)} = M$.
Accordingly, the singular chain complex $S_*(M)$ has an increasing filtration: $\{ 0\} \subset S_*(M^{(0)}) \subset S_*(M^{(1)}) =S_*(M)$.
The associated spectral sequence is the one that we consider.

Now let $\alpha \in H_1(V_g)^{\pi_1(S^1)}$.
There is an isomorphism
\[
E^2_{1,1} = H_1(S^1; \HH_1(V_g)) 
 \cong \Ker(\partial_* \colon H_2(M,M^{(0)}) \to H_1(M^{(0)}) ),
\]
under which the homology class $[e^1 \otimes \alpha]$ is mapped to the homology class of the relative $2$-cycle $\Phi_{\sharp}(I \times \tilde{a})$.
However, since $e^0 \times \sigma_{\varphi,\alpha} \in S_2(M^{(0)})$, it holds that
\[
[\Phi_{\sharp}(I \times \tilde{a})]
= [\Phi_{\sharp}(I \times \tilde{a}) - e^0 \times \sigma_{\varphi,\alpha}] = [\mathcal{Z}_{\alpha}]
\in H_2(M,M^{(0)}).
\]
Thus the homology class under consideration is now represented by a \emph{genuine} $2$-cycle in $M$.
Finally, we observe that the natural map
\[
H_2(M) \cong E^{\infty}_{1,1}
\overset{\cong}{\longrightarrow} E^2_{1,1} \subset H_2(M,M^{(0)})
\]
coincides with the inclusion homomorphism.
This completes the proof.
\end{proof}

By Lemma~\ref{lem:aZa}, it is enough to compute $i_*([\mathcal{Z}_{\alpha}])$.
Since $\tilde{a}$ is a $1$-cycle in $\Sigma_g = \partial V_g$, the $2$-chain $\Phi_{\sharp}(I \times \tilde{a})$ lies in $\partial M_{\varphi}$.
Hence
\[
\mathcal{Z}_{\alpha} = - e^0 \times \sigma_{\varphi,\alpha}
\in Z_2(M_{\varphi},\partial M_{\varphi}).
\]
This shows that $i_*([\mathcal{Z}_{\alpha}])$ is represented by the relative $2$-cycle $-e^0 \times \sigma_{\varphi,\alpha}$.
This completes the proof of Proposition~\ref{prop:incl}. 

%\subsection{Intersection form} \label{subsec:intform}
\subsection{Proof of Theorem~\ref{thm:main1}} \label{subsec:intform}

We describe the intersection form of $M_{\varphi}$ and prove Theorem~\ref{thm:main1}.

First we claim that the second homology group $H_2(M_{\varphi})$ is naturally isomorphic to $U_{\varphi}^{\Z}:= \Ker(S-I_g) \subset \Z^g$.
In fact, by Lemma~\ref{lem:H2abs} we have $H_2(M_{\varphi}) \cong H_1(V_g)^{\pi_1(S^1)}$, and the action of $\varphi$ on $H_1(V_g) \cong \mathbb{Z}^g$ is given by the matrix $S$.
Thus the claim follows.

We next claim that under the isomorphism $H_2(M_{\varphi}) \cong U_{\varphi}^{\Z}$, the intersection form on $H_2(M_{\varphi})$ is transferred to the bilinear form $\langle \ ,\ \rangle_{\varphi}$.
Since $\phi_g^V(\varphi) = \Sign M_{\varphi}$, this will complete the proof of Theorem~\ref{thm:main1}.
The proof of this claim consists of two steps.

\emph{Step 1.}
We give a description of the bilinear form on $H_1(V_g)^{\pi_1(S^1)}$ that is obtained by transferring the intersection form on $H_2(M_{\varphi})$.
%under the isomorphism $H_2(M_{\varphi}) \cong H_1(V_g)^{\pi_1(S^1)}
Let
$
\langle \cdot, \cdot \rangle_V \colon
H_2(V_g,\Sigma_g) \times H_1(V_g) \to \mathbb{Z}
$
be the intersection product of the compact oriented $3$-manifold $V_g$.
We have
\begin{equation}
\label{eq:Db}
\langle D_i, \beta_j \rangle_V = \delta_{ij} \quad
\text{for any $i,j \in \{ 1,\ldots,g\}$}.
\end{equation}
Let
\begin{equation}
\label{eq:H_0H_1}
H_0(S^1; \HH_2(V_g,\Sigma_g)) \times H_1(S^1; \HH_1(V_g)) \longrightarrow \Z
\end{equation}
be the intersection product of $H_0(S^1; \HH_2(V_g,\Sigma_g))$ and  $H_1(S^1; \HH_1(V_g))$ followed by the contraction of the coefficients by the form $\langle \cdot, \cdot \rangle_V$.
Under the isomorphisms in Lemmas \ref{lem:H2abs} and \ref{lem:H2rel}, this is equivalent to the intersection product $H_2(M_{\varphi})\times H_2(M_{\varphi},\partial M_{\varphi})\to\mathbb{Z}$.
By composing \eqref{eq:H_0H_1} and the homomorphism
\begin{align*}
H_1(V_g)^{\pi_1(S^1)} \times H_1(V_g)^{\pi_1(S^1)}
& \overset{d\otimes \id}{\longrightarrow}
H_2(V_g,\Sigma_g)_{\pi_1(S^1)} \times H_1(V_g)^{\pi_1(S^1)} \\
& \hspace{0.4em} \cong \hspace{0.4em} H_0(S^1; \HH_2(V_g,\Sigma_g)) \times H_1(S^1; \HH_1(V_g)),
\end{align*}
we obtain a bilinear form on $H_1(V_g)^{\pi_1(S^1)}$.
Proposition \ref{prop:incl} implies that this is equivalent to the intersection form on $H_2(M_{\varphi})$.
%We show that it is represented by the matrix ${}^tQ$ below.

\emph{Step 2.}
We prove that the bilinear form on $H_1(V_g)^{\pi_1(S^1)}$ described in the previous paragraph is equivalent to $\langle \ ,\ \rangle_{\varphi}$ under the identification $H_1(V_g)^{\pi_1(S^1)} \cong U_{\varphi}^{\Z}$.
Let $x = (x_1,\ldots,x_g)$, $y=(y_1,\ldots,y_g) \in U_{\varphi}^{\Z} \subset \Z^g$.
We regard $x$ as an element of $H_1(V_g)^{\pi_1(S^1)}$.
Then, we can take $\tilde{x} = \sum_{i=1}^g x_i \beta_i \in H_1(\Sigma_g)$ as a lift of $x$ which we need to compute $d(x)$.
Thus we have
\[
\varphi_* (\tilde{x}) -\tilde{x} =
(\alpha_1,\ldots,\alpha_g)\, Q\, {}^t (x_1,\ldots,x_g)
= (x_1,\ldots,x_g)\, {}^t  Q\, {}^t (\alpha_1,\ldots,\alpha_g),
\]
and hence
$
d(x) = (x_1,\ldots,x_g)\, {}^t  Q\, {}^t (D_1,\ldots,D_g).
$
Therefore, the pairing of $x$ and $y$ by the bilinear form on $H_1(V_g)^{\pi_1(S^1)}$ described above is equal to
\[
\big\langle
(x_1,\ldots,x_g)\, {}^t  Q\, {}^t (D_1,\ldots,D_g),
(\beta_1,\ldots,\beta_g)\, {}^t (y_1,\ldots,y_g)
\big\rangle_V
= {}^t x\,  {}^t Q\, y
= \langle x ,y \rangle_{\varphi}.
\]
Here we used the equality \eqref{eq:Db}.
This completes the proof of Theorem~\ref{thm:main1}.

\begin{remark}\label{rem:GilMas}
There is a $2$-cocycle $m_{\lambda}$ on $\Sp(2g;\mathbb{Z})$ constructed by Turaev~\cite{Turaev} which satisfies $[m_{\lambda}]=-[\tau_g]\in H^2(\Sp(2g;\Z))$,
and
Walker, in page 124 of his note\footnote{K.~Walker (1991). \emph{On Witten's $3$-manifold invariants}, Preliminary Version [online]. Website https://canyon23.net/math/1991TQFTNotes.pdf [accessed 1 May 2020].},
constructed a (unique) cobounding function
$j\colon \mcg\to\Q$ of the sum $\rho^*\tau_g+\rho^*m_{\lambda}$ of $2$-cocycles.
The $2$-cocycle $m_{\lambda}$ and the function $j$ depend on the choice of a lagrangian $\lambda \subset H_1(\Sigma_g;\Q)$.
If we choose a suitable lagrangian $\lambda$,
the restriction of $j$ to $\hb$ is known to be a cobounding function of $\rho^*\tau_g$,
and coincides with our function $\phi_g^V$.
Gilmer and Masbaum~\cite[Proposition~6.9]{GilMas} described $j$ explicitly in a way which is similar to but different from ours. 
\end{remark}

\begin{remark}
\label{rem:symmetric}
Since $Sy = y$ for any $y\in U_{\varphi}$, we have $\langle x, y\rangle_{\varphi}= {}^t x\,  {}^t Q S\, y$ for any $x,y \in U_{\varphi}$.
Since ${}^t Q S$ is symmetric by \eqref{eq:PQS}, this gives a purely algebraic explanation for the symmetric property of the form $\langle \ , \ \rangle_{\varphi}$ on $U_{\varphi}$.
\end{remark}

\begin{remark}
\label{rem:phiur}
By Theorem~\ref{thm:main1}, one can regard $\phi^V_g$ as a $1$-cochain on $\urSp(2g;\Z)$.
For $g\ge 3$, it is the unique $1$-cochain which cobounds $\tau_g$ on $\urSp(2g;\Z)$ since $H^1(\urSp(2g;\Z))=0$; see \cite[Corollary 4.4]{Sakasai}.
\end{remark}

\section{Evaluation of Meyer functions}\label{section:Meyer function}

\subsection{The Meyer function on the hyperelliptic mapping class group} \label{subsec:phiH}

There is a unique $1$-cochain $\phi_g^{\mathcal{H}} \colon \mathcal{H}(\Sigma_g) \to \Q$ such that for any $\varphi_1,\varphi_2 \in \mathcal{H}(\Sigma_g)$,
\begin{equation}
\label{eq:phiH}
\phi_g^{\mathcal{H}}(\varphi_1) + \phi_g^{\mathcal{H}}(\varphi_2) 
-\phi_g^{\mathcal{H}}(\varphi_1\varphi_2) = \tau_g (\rho(\varphi_1), \rho(\varphi_2)).
\end{equation}
The $1$-cochain $\phi_g^{\mathcal{H}}$ is called the \emph{Meyer function on the hyperelliptic mapping class group of genus $g$};
see \cite{Endo00, Morifuji03}.

Recall the element $s_1 = t_2 t_3 t_1 t_2 \in \mathcal{H}(V_g)\subset \mathcal{H}(\Sigma_g)$ which was defined in Section~\ref{subsec:hhb}.
\begin{lemma}
\label{lem:phiHs}
$\phi_g^{\mathcal{H}}(s_1) = (2g + 3)/(2g + 1)$.
\end{lemma}

\begin{proof}
Set $T_i = \rho(t_i)$ for every $i\in \{1,2,3\}$.
Using \eqref{eq:phiH}, we have
\begin{align*}
\phi_g^{\mathcal{H}}(s_1) &= \phi_g^{\mathcal{H}}(t_2)
+ \phi_g^{\mathcal{H}}(t_3) + \phi_g^{\mathcal{H}}(t_1) + \phi_g^{\mathcal{H}}(t_2) \\
& \hspace{1em} - \tau_g(T_1,T_2) - \tau_g(T_3,T_1T_2) - \tau_g(T_2,T_3T_1T_2).
\end{align*}
As was shown in \cite[Lemma 3.3]{Endo00} and \cite[Proposition 1.4]{Morifuji03},
we have $\phi_g^{\mathcal{H}}(t_i) = (g + 1)/(2g + 1)$ for all $i \in \{1,2,3\}$.
Also, by a direct computation we obtain $\tau_g(T_1,T_2) = 0$, $\tau_g(T_3, T_1T_2) = 0$, and $\tau_g(T_2,T_3T_1T_2) = 1$.
The result follows from these equalities.
\end{proof}

\subsection{The Meyer function on the handlebody group} \label{subsec:phiV}

Recall from the introduction that we defined $\phi_g^V\colon\hb\to\mathbb{Z}$ by $\varphi\mapsto\Sign M_{\varphi}$,
where $M_{\varphi}$ is the mapping torus of $\varphi$.
\begin{lemma}\label{lem:cobounding}
The function $\phi^V_g\colon\hb\to\mathbb{Z}$ cobounds the cocycle $\rho^*\tau_g$ in the handlebody group $\hb$.
If $g\ge 3$, $\phi^V_g$ is the unique cobounding function of $\rho^*\tau_g$. 
\end{lemma}
\begin{proof}
The uniqueness follows from the fact that $H_1(\hb)$ is torsion when $g\ge3$.

For given two mapping classes $\varphi, \psi \in \hb$, there is an oriented $V_g$-bundle $W(\varphi,\psi) \to P$ such that the monodromy along $\ell_1$, $\ell_2$ and $\ell_3$ are $\varphi$, $\psi$ and $(\varphi\psi)^{-1}$, respectively.
The boundary of $W(\varphi,\psi)$ is written as 
\[
\partial W(\varphi, \psi)=E(\varphi,\psi)\cup(M_{\varphi^{-1}}\sqcup M_{\psi^{-1}}\sqcup M_{\varphi\psi}).
\]
Note that $M_{\varphi^{-1}}$ is diffeomorphic to $-M_{\varphi}$ under an orientation-preserving diffeomorphism,
where $-M_{\varphi}$ denotes the mapping torus $M_{\varphi}$ with orientation reversed.
Since the signature of $\partial W(\varphi,\psi)$ is zero, Novikov additivity implies that
\[
\Sign E(\varphi,\psi)-\Sign M_{\varphi}-\Sign M_{\psi}+\Sign M_{\varphi\psi}=0.
\]
This shows that $\phi^V_g$ is a cobounding function of $\rho^*\tau_g$ restricted to $\hb$.
\end{proof}

Since $\dim V_{A,B}\le 4g$ for any $A,B\in\Sp(2g;\Z)$,
the signature cocycle $\tau_g$ is a bounded 2-cocycle.
Therefore, it represents a class in the second bounded cohomology group $H_b^2(\mcg)$.
The image of $[\tau_g]$ under the natural homomorphism
$H_b^2(\mcg;\mathbb{Q})\to H_b^2(\hyp;\mathbb{Q})$ is non-trivial since the Meyer function $\phi_g^{\mathcal{H}}$ is unbounded.
In contrast, we have:
\begin{proposition}
\label{prop:bddvan}
Under the natural homomorphism
$H_b^2(\mcg;\mathbb{Q})\to H_b^2(\hb;\mathbb{Q})$,
the image of the cohomology class $[\tau_g]$ vanishes.
\end{proposition}
\begin{proof}
The restriction of the signature cocycle $\tau_g$ to $\hb$ is cobounded by the function $\phi_g^V$,
and $\phi_g^V$ is a bounded function since the rank of $H_2(M_{\varphi})$ is at most $g$.
\end{proof}

\subsection{Computation of the Meyer function on the handlebody group} \label{subsec:comphiV}

Theorem~\ref{thm:main1} shows that the bilinear form $\langle \ ,\ \rangle_{\varphi}$ on $U_{\varphi}$, whose signature coincides with $\phi_g^V(\varphi)$, can be computed from the homological monodromy $\rho(\varphi)\in \urSp(2g;\Z)$.
In more detail, if $\rho(\varphi) = \begin{pmatrix} P & Q \\ O_g & S \end{pmatrix}$, then $U_{\varphi} = \Ker(S-I_g) \subset \Q^g$ and $\langle x, y \rangle_{\varphi} = {}^t x\,  {}^t Q\, y$ for $x,y \in U_{\varphi}$.

The $1$-cochain $\phi_g^V$, regarded as the one defined on $\urSp(2g;\Z)$, is \emph{stable} with respect to $g$ in the following sense.
For every non-negative integer $g\ge 0$, there is a natural embedding $\iota\colon \urSp(2g;\Z) \hookrightarrow \urSp(2(g+1);\Z)$;
\[
A= \begin{pmatrix} P & Q \\ O_g & S \end{pmatrix}
\mapsto \iota(A) =
\begin{pmatrix} \tilde{P} & \tilde{Q} \\ O_{g+1} & \tilde{S} \end{pmatrix},
\]
where
\[
\tilde{P} = \begin{pmatrix} P & 0 \\ 0 & 1 \end{pmatrix}, \quad
\tilde{Q} = \begin{pmatrix} Q & 0 \\ 0 & 0 \end{pmatrix}, \quad
\tilde{S} = \begin{pmatrix} S & 0 \\ 0 & 1 \end{pmatrix}.
\]
Then $\phi_{g+1}^V(\iota(A)) = \phi_g^V(A)$ for any $A\in \urSp(2g;\Z)$.

\begin{lemma}
\label{lem:phiVt}
For any positive integer $m$, we have $\phi_g^V(t_1^m) = 1$.
\end{lemma}

\begin{proof}
Since the action of $\rho(t_1^m)$ on $H_1(\Sigma_g)$ is given by
\[
\rho(t_1)\colon
\alpha_i \mapsto \alpha_i \quad (i = 1,\ldots,g),
\quad \quad
\beta_1 \mapsto m \alpha_1 + \beta_1,
\quad \quad
\beta_i \mapsto \beta_i \quad (i = 2,\ldots,g),
\]
we may assume that $g=1$.
Then $\rho(t_1^m) = \begin{pmatrix} 1 & m \\ 0 & 1 \end{pmatrix}$, and $\Ker(S-I_1) = \Z$ on which the pairing is given by the $1\times 1$ matrix $\begin{pmatrix} m \end{pmatrix}$. 
Hence $\phi_g^V(t_1^m)=1$, as required.
\end{proof}

\begin{lemma}
\label{lem:phiVs}
$\phi^V_g(s_1) = 1$.
\end{lemma}

\begin{proof}
The proof proceeds as in the same way as the previous lemma.
In this case we may assume that $g = 2$.
Then
\[
\rho(s_1) = \begin{pmatrix} P & Q \\ O_2 & S \end{pmatrix}
\quad
\text{with}
\quad
P= \begin{pmatrix} -1 & 0 \\ 1 & 1 \end{pmatrix}, \quad
Q= \begin{pmatrix} 2 & -1 \\ -1 & 1 \end{pmatrix}, \quad
S= \begin{pmatrix} -1 & 1 \\ 0 & 1 \end{pmatrix}.
\]
The rest of computation is straightforward, so we omit it.
\end{proof}

%When $g\ge3$, the cohomology class $[\tau_g]$ is
%4 times a generator of $H^2(\mcg)$.
%Since $\Im\phi_g^V$ contains $1\in\mathbb{Z}$,
%we obtain:

%\begin{proposition}
%\label{prop:Z4}
%When $g\ge3$,
%the kernel of the natural homomorphism $H^2(\mcg)\to H^2(\hb)$ is generated by $[\tau_g]$,
%and there exists an injective homomorphism $\mathbb{Z}/4\mathbb{Z}\to H^2(\hb)$.
%\end{proposition}

\subsection{Proof of Theorem~\ref{thm:main}} \label{subsec:pr11}

Since both the $1$-cochains $\phi_g^{\mathcal{H}}$ and $\phi_g^V$ cobound the signature cocycle, their difference becomes a $\Q$-valued homomorphism on $\mathcal{H}(V_g) = \mathcal{H}(\Sigma_g) \cap \hb$.

We compare the homomorphism $\phi_g^{\mathcal{H}} - \phi_g^V$ with the generator $\mu \in H^1(\mathcal{H}(V_g))$ in Corollary~\ref{cor:mu}.
It is sufficient to evaluate $\phi_g^{\mathcal{H}} - \phi_g^V$ on $s_1$ if $g$ is even, and on $t_1 s_1^{\frac{g+1}{2}}$ if $g$ is odd.
By Lemmas~\ref{lem:phiHs} and \ref{lem:phiVs} we immediately obtain
\begin{equation}
\label{eq:phiHVs}
(\phi_g^{\mathcal{H}} - \phi_g^V)(s_1) = \frac{2}{2g + 1}.
\end{equation}
This settles the case where $g$ is even.
When $g$ is odd, we compute
\begin{align*}
(\phi_g^{\mathcal{H}} - \phi_g^V)(t_1 s_1^{\frac{g + 1}{2}})
&= (\phi_g^{\mathcal{H}} - \phi_g^V)(t_1)
+ \frac{g + 1}{2} (\phi_g^{\mathcal{H}} -\phi_g^V)(s_1) \\
&= \left( \frac{g + 1}{2g + 1}-1 \right) + \frac{g + 1}{2} \cdot \frac{2}{2g + 1} \\
& = \frac{1}{2g + 1}.
\end{align*}
Here, we used the fact that $\phi_g^{\mathcal{H}} - \phi_g^V$ is a homomorphism on $\mathcal{H}(V_g)$ in the first line; we used the fact that $\phi_g^{\mathcal{H}}(t_1) = (g+1)/(2g+1)$ (see the proof of Lemma \ref{lem:phiHs}), Lemma~\ref{lem:phiVt} and \eqref{eq:phiHVs} in the second line.
This completes the proof of Theorem~\ref{thm:main}. 

\vspace{1em}
\noindent {\bf Acknowledgments.} The authors would like to thank Susumu Hirose for his helpful comments.
Y.~K.~ is supported by JSPS KAKENHI 18K03308.
M.~S.~ is supported by JSPS KAKENHI 18K03310.

\end{document}